\documentclass[12pt]{amsart}
\usepackage{amsmath,palatino}
\usepackage{amsthm}
\usepackage{enumerate}
\usepackage{amssymb}
\usepackage[utf8]{inputenc}

\input xy
\xyoption{all}

\title{Geometric structure of Dimension functions of certain Continuous fields}
\author{Ramon Antoine} 
\author{Joan Bosa}
\author{Francesc Perera}
\author{Henning Petzka}
\address{Departament de Matem\`atiques, Universitat Aut\`onoma de Barcelona, 08193 Bellaterra, Barcelona, Spain}\email{ramon@mat.uab.cat, jbosa@mat.uab.cat, perera@mat.uab.cat, \newline\indent petzka@mat.uab.cat}

\date{\today}

\theoremstyle{plain}
\newtheorem{lemma}{Lemma}[section]
\newtheorem{theorem}[lemma]{Theorem}
\newtheorem{corollary}[lemma]{Corollary}
\newtheorem{proposition}[lemma]{Proposition}
\newtheorem{definition}[lemma]{Definition}
\newtheorem*{proposition*}{Proposition}
\newtheorem*{theorem*}{Theorem}
\newtheorem*{definition*}{Definition}
\newtheorem*{claim*}{Claim}
\newtheorem*{notation*}{Notation}

\newtheorem{remark}[lemma]{Remark}

\newcommand{\Cu}{\mathrm{Cu}}
\newcommand{\PreCu}{\mathrm{PreCu}}

\newcommand{\ra}{\rangle}
\newcommand{\la}{\langle}

\begin{document}

\begin{abstract}
In this paper we study structural properties of the Cuntz semigroup and its functionals for continuous fields of C$^*$-algebras over finite dimensional spaces. In a variety of cases, this leads to an answer to a conjecture posed by Blackadar and Handelman. Enroute to our results, we determine when the stable rank of continuous fields of C$^*$-algebras over one dimensional spaces is one.
\end{abstract}

\maketitle
 \section*{Introduction}

%%%%%%%%%%%%
Recent years of research aiming at the classification of C$^*$-algebras demonstrated the Cuntz semigroup $\mathrm{W}(A)$, an  invariant built out of the positive elements in matrix algebras over a C$^*$-algebra $A$, to be an important ingredient. Cuntz \cite{Cu} initiated the study of the functionals on this ordered semigroup, i.e. the normalized semigroup homomorphisms into the positive reals respecting the order. These functionals are referred to as dimension functions of $A$.

Blackadar and Handelman posed two conjectures on the geometry  of the set of dimension functions on a given C$^*$-algebra $A$ in their 1982 paper \cite{BH}. Firstly, they conjectured that the set of dimension functions forms a simplex. Secondly, they conjectured that the set of lower semicontinuous dimension functions is dense in the set of all dimension functions. The relevance of the latter conjecture lies in the fact that the set of lower semicontinuous dimension functions, being in correspondence with the quasitraces in $A$, is more tractable than the set of all dimension functions. 

The second named conjecture was proved in \cite{BH} for commutative C$^*$-algebras, while the first conjecture was left completely unanswered. In \cite{Per}, the first conjecture was proved for unital C$^*$-algebras with stable rank one and real rank zero. Both conjectures were shown to hold for all unital, simple, separable, exact and $\mathcal{Z}$-stable C$^*$-algebras in \cite{BPT}, by giving a suitable representation of their Cuntz semigroup. Further, the first of the above conjectures was strengthened in \cite{BPT} to asking the set of dimension functions to form a Choquet simplex.

Taking advantage of recent advances in the computation of $\mathrm{W}(A)$ for certain C$^*$-algebras $A$, we show in this paper how we were able to answer the Blackadar-Handelman conjectures affirmatively
for certain continuous fields of C$^*$-algebras. Along the way, we obtain independently interesting results on the stable rank of continuous fields completing accomplishments by Nagisa, Osaka and Phillips in \cite{nop}. 

We would like to highlight at this point the key steps of how we were able to prove the set of dimension functions to be a Choquet simplex for certain continuous fields of C$^*$-algebras. 
Performing Grothendieck's construction on $\mathrm{W}(A)$ gives an ordered abelian group $\mathrm{K}_0^*(A)$. By a result in \cite{Goo}, the set of dimension functions forms a Choquet simplex provided $\mathrm{K}_0^*(A)$ has Riesz interpolation. An affirmative answer to the first conjecture is then linked to certain structural properties of the Cuntz semigroup. Thus, it appears to be interesting to determine when $\mathrm{K}_0^*(A)$ has interpolation for a general C$^*$-algebra $A$. 

It is known that $\mathrm{K}_0^*(A)$ has interpolation if $\mathrm{W}(A)$ does, and, as a general strategy in our setting, we concentrate on proving the latter. We show that $\mathrm{W}(A)$ has interpolation if and only if its stabilized version $\Cu(A)=\mathrm{W}(A\otimes\mathbb{K})$ (see \cite{CEI}) has interpolation, provided that $\mathrm{W}(A)$ lies naturally as a hereditary subsemigroup
in $\Cu(A)$. We prove for certain continuous fields of C$^*$-algebras $A$ that $\Cu(A)$ has interpolation. Since it is known that the inclusion $\mathrm{W}(A)\subseteq \Cu(A)$ is hereditary whenever $A$ has stable rank one, we are led at this point to the question on when continuous fields of C$^*$-algebras have stable rank one. Building on work of \cite{nop}, we settle this question in great generality: The algebra $\mathrm{C}(X,D)$ of continuous functions from a one-dimensional compact metric space into a C$^*$-algebra $D$ has stable rank one if and only if the stable rank of $D$ is one and every hereditary subalgebra $B$ of $D$ has trivial $\mathrm{K}_1$. For general continuous fields of C$^*$-algebras $A$ the condition requiring each fiber to have no $\mathrm{K}_1$-obstructions (in the above sense) is still a sufficient condition for $\mathrm{sr}(A)=1$, but it is not a necessary condition in the general case. Therefore, for certain continuous fields $A$ of stable rank one, 
we prove that $\mathrm{K}_0^*(A)$ has interpolation and this confirms the first conjecture.

Our approach to the second conjecture is based on representing $\mathrm{K}_0^*(A)$ sufficiently well into the group of affine and bounded functions on the trace space of $A$.

The paper is organized as follows. Section 1 contains our results on the stable rank of continuous fields, which can be read independently of the rest of the paper. It follows a short section on hereditariness of $\mathrm{W}(A)$ in $\Cu(A)$ for continuous fields of C$^*$-algebras $A$. Interpolation results are proved in Section 3, before we apply our results in Section 4 to answer the Blackadar-Handelman conjectures affirmatively for the C$^*$-algebras under consideration.

\section{Continuous fields of stable rank one}\label{sr}

Let $X$ be a compact metric space of dimension one. We will prove in this section that the algebra $\mathrm{C}(X,A)$ of continuous functions from $X$ into a C$^*$-algebra $A$, has stable rank one, if $A$ has no $\mathrm{K}_1$-obstructions as defined below. We also prove the converse direction in a setting of great generality and prove, as an application,  corresponding results for continuous fields. 

Recall from \cite{ABP2}, \cite{APS}, that a C$^*$-algebra $A$ has \emph{no $\mathrm{K}_1$-obstructions} provided that $A$ has stable rank one and $\mathrm{K}_1(B)=0$ for every hereditary subalgebra $B$ of $A$ (equivalently, $\mathrm{sr}(A)=1$ and $\mathrm{K}_1(I)=0$ for every closed two-sided ideal of $A$). In the case that $A$ is simple or, as proved by Lin, if $A$ has real rank zero (see \cite[Lemma 2.4]{Li}), $A$ has no $\mathrm{K}_1$-obstructions if and only if $\mathrm{K}_1(A)=0$.

We start by considering the case where $X$ is the closed unit interval. It was already noted in \cite{nop} that $\mathrm{sr}(A)=1$ and $\mathrm{K}_1(A)=0$ are necessary conditions for $\mathrm{sr}(C([0,1],A))=1$. The first condition follows from the fact that $A$ is a quotient of $C([0,1],A)$, the second is \cite[Proposition 5.2]{nop}. We show that also every hereditary subalgebra $B$ must have trivial $\mathrm{K}_1$.

\begin{proposition}\label{NecessaryConditions}
Let $A$ be any C$^*$-algebra. If $\mathrm{sr}(C([0,1],A))=1$, then $A$ has no $\mathrm{K}_1$-obstructions.
\end{proposition}

\begin{proof}
As already noted previous to the proposition, it is clear that $\mathrm{sr}(A)=1$ is a necessary condition.

Let $B\subseteq A$ be a hereditary subalgebra. Let $I$ denote the ideal generated by $B$. Then $C([0,1],I)$ is an ideal of $C([0,1],A)$ and therefore has stable rank one. It follows from \cite[Proposition 5.2]{nop} that $\mathrm{K}_1(I)=0$. Since $B$ is a full hereditary subalgebra of $I$, we conclude that $\mathrm{K}_1(B)=0$.
\end{proof} 

We will show that conversely, for any C$^*$-algebra $A$ with no $\mathrm{K}_1$-obstructions the stable rank of $C([0,1],A)=1$. Our proof will follow the proof of \cite[Theorem 4.3]{nop}, where the same result was shown to hold for C$^*$-algebras $A$ with $\mathrm{sr}(A)=1$, $RR(A)=0$, and $\mathrm{K}_1(A)=0$. The proof of \cite[Theorem 4.3]{nop} refers to Lemma 4.2 of the same paper. Our contribution is to prove the corresponding lemma to hold in a more general setting.

\begin{lemma}\label{InvertiblePath}
Let $A$ be a unital C$^*$-algebra with no $\mathrm{K}_1$-obstructions.

For any given $\epsilon>0$ there is some $\delta>0$ such that whenever $a$ and $b$ are two invertible contractions in $A$ with $\|a-b\|<\delta$ then there is a continuous path $(c_t)_{t\in[0,1]}$ in the invertible elements of $A$ such that $c_0=a$, $c_1=b$, and $\|c_t-a\|<\epsilon$ for all $t\in[0,1]$.
\end{lemma}

\begin{proof}
For given $\epsilon >0$ we choose $\delta_0>0$ satisfying the conclusion of \cite[Lemma 3.4]{Lsant} for $\frac{\epsilon}{2}$, i.e., for any positive contraction $a$ and any unitary $u$ with $\|ua-a\|<\delta_0$ there is a path of unitaries $(u_t)_{t\in[0,1]}$ in $A$ such that $u_0=u$, $u_1=1_A$, and $\|u_ta-a\|<\frac{\epsilon}{2}$ for all $t\in[0,1]$. (It follows from our assumptions and \cite[Theorem 2.10]{Rie} that $U(B^\sim)$ is connected for each hereditary subalgebra $B$ of $A$, which is needed for the application of \cite[Lemma 3.4]{Lsant}.) Find $0<\delta\leq \frac{\delta_0}{2}$ such that whenever $\|a-b\|<\delta$, then $\| |a|-|b| \|<\frac{\delta_0}{2}$. (This is possible by Lemma 2.8 of \cite{Li2}.)

Take two invertible contractions $a,b$ in $A$ with $\|a-b\|<\delta$ and write $a=u|a|$ and $b=v|b|$ with unitaries $u,v\in A$. We first connect $a$ and $u|b|$ by a path of invertible elements. To do this, define a continuous path $(w_t)_{t\in[0,1]}$ by
$$w_t:=u(t|b|+(1-t)|a|),\ t\in[0,1].$$
Then $w_0=a$, $w_1=u|b|$ and, for any $t\in [0,1]$, $w_t$ is invertible and 
$$\|w_t-a\|=\|ut|b|-ut|a|\|= t\| |b|-|a|\|<\frac{\delta_0}{2}<\epsilon.$$

Next, we connect $u|b|$ and $b$ by a path of invertible elements. Since
$$\|v^*u|b|-|b|\|=\|u|b|-v|b|\|\leq \|u|b|-u|a|\|+\|u|a|-v|b|\|=\||a|-|b|\|+\|a-b\|<\delta_0,$$
an application of \cite[Lemma 3.4]{Lsant} provides us with a path of unitaries $(u_t)_{t\in[0,1]}$ in $A$ such that $u_0=v^*u$, $u_1=1$, and $\|u_t|b|-|b|\|<\frac{\epsilon}{2}$ for all $t\in[0,1]$.

Define a continuous path $(z_t)_{t\in[0,1]}$ by
$$z_t:=vu_t|b|,\ t\in[0,1].$$
Then $z_0=u|b|$, $z_1=b$ and, for each $t\in[0,1]$, $z_t$ is invertible and
$$\|z_t-a\|\leq \|vu_t|b|-v|b|\|+\|v|b|-a\|=\|u_t|b|-|b|\|+\|b-a\|<\epsilon.$$

Hence $$c_t:=\left \{\begin{array}{ll}w_{2t},&t\in[0,\frac{1}{2}]\\ z_{2t-1},&t\in[\frac{1}{2},1]\end{array}\right.$$
is the continuous path with the desired properties.
\end{proof}

\begin{theorem}\label{StableRankOfInterval}
Let $A$ be any C$^*$-algebra with $\mathrm{sr}(A)=1$. Then 
$$\mathrm{sr}(C([0,1],A))=\left \{ \begin{array}{ll} 1,&\mbox{ if $A$ has no $\mathrm{K}_1$-obstructions}\\ 2,& else.\end{array}\right.$$ 
\end{theorem}

\begin{proof}
It is known that $\mathrm{sr}(C([0,1],A))\leq 1+\mathrm{sr}(A)\leq 2$ (\cite{Su}). From Proposition \ref{NecessaryConditions} we know that for $\mathrm{sr}(C([0,1],A))=1$ it is a necessary condition that $\mathrm{K}_1(B)=0$ for all hereditary subalgebras $B$ of $A$. To show that this condition is also sufficient we follow the lines of the proof of \cite[Theorem 4.3]{nop}, applying Lemma \ref{InvertiblePath} instead of \cite[Lemma 4.2]{nop}.
\end{proof}

\begin{corollary}\label{StableRankForSimple}
Let $A$ be a simple C$^*$-algebra with $\mathrm{sr}(A)=1$ and $\mathrm{K}_1(A)=0$. Then \[ \mathrm{sr}(C([0,1],A))=1\,.\]
\end{corollary}

The previous corollary answers positively a question raised in \cite[5.9]{nop}, and also allows for a much simpler proof of \cite[Theorem 5.7]{nop} for Goodearl algebras, since these are always simple and have stable rank one (see \cite{Goo3}), as follows:

\begin{corollary}
Let $A$ be a Goodearl algebra with $\mathrm{K}_1(A)=0$. Then $\mathrm{sr}(\mathrm{C}([0,1],A))=1$.
\end{corollary}

Another application of Theorem \ref{StableRankOfInterval} is the computation of the stable rank of tensor products $A\otimes\mathcal Z$ of C$^*$-algebras $A$ with no $\mathrm{K}_1$-obstructions with the Jiang-Su algebra $\mathcal{Z}$. This was proved by Sudo in \cite[Theorem 1.1]{sudo2} assuming that $A$ has real rank zero, stable rank one, and trivial $\mathrm{K}_1$.

\begin{corollary}
Let $A$ be a C$^*$-algebra with no $\mathrm{K}_1$-obstructions. Then the stable rank of $A\otimes\mathcal{Z}$ is one.
\end{corollary}

\begin{proof}
It is well-known that we can write $A\otimes \mathcal{Z}$ as an inductive limit
$$A\otimes\mathcal{Z}=\lim_{i\rightarrow \infty}A\otimes Z_{p_i,q_i},$$
with pairs of co-prime numbers $(p_i,q_i)$ and prime dimension drop algebras $$Z_{p_i,q_i}=\{f\in C([0,1],M_{p_i}\otimes M_{q_i}\ |\ f(0)\in I_{p_i}\otimes M_{q_i},\ f(1)\in M_{p_i} \otimes I_{q_i} \}.$$
Since the stable rank of inductive limit algebras satisfies that $\mathrm{sr}(\lim_{i\rightarrow\infty}(A_i))\leq \liminf \mathrm{sr}(A_i)$ (\cite{Rie2}) it suffices to show that the stable rank of each $Z_{p_i,q_i}\otimes A$ is one.

Fix two co-prime numbers $p$ and $q$ and write $Z_{p,q}\otimes A$ as a pullback
$$\xymatrix{Z_{p,q}\otimes A\ar@{-->}[r]\ar@{-->}[d] & M_p(A)\oplus M_q(A) \ar[d]^\phi \\ C([0,1], M_{pq}(A)) \ar@{->>}[r]^{(\lambda_0,\lambda_1)} & M_{pq}(A)\oplus M_{pq}(A)}$$
with maps $\lambda_i(f)=f(i)$ and $\phi(A,B)=\left ( A\otimes I_q, I_p\otimes B\right )$. 

Our assumptions together with Theorem \ref{StableRankOfInterval} imply that $\mathrm{sr}(C([0,1],M_{pq}(A)))=1$. Further, $\mathrm{sr}(M_{m}(A))=1$ for all $m\in\mathbb{N}$, and the map from left to right in the pullback diagram is surjective. An application of \cite[Theorem 4.1 (i)]{bp} implies that $\mathrm{sr}(Z_{p,q}\otimes A)=1$.
\end{proof}

We now turn our attention to $\mathrm{C}(X,A)$ for compact metric spaces $X$ with $\dim(X)=1$. From Theorem \ref{StableRankOfInterval} it follows that the stable rank of $\mathrm{C}(X,A)$ is one whenver the stable rank of $C([0,1],A)$ is one.

\begin{corollary}\label{CorDimOne}
Let $A$ be a separable C$^*$-algebra $A$ with no $\mathrm{K}_1$-obstructions, and let $X$ be a compact metric space of dimension one. Then $\mathrm{sr}(\mathrm{C}(X,A))=1$.
\end{corollary}

\begin{proof}
If $X$ is a finite graph with $m$ edges and $V$ is its set of vertices, then $\mathrm{C}(X,A)$ can be written as a pullback 
$$\xymatrix{\mathrm{C}(X,A)\ar@{-->}[r]\ar@{-->}[d] & C([0,1],A^m) \ar[d] \\  A^n \cong C(V,A) \ar[r] & A^m\oplus A^m}$$
with maps $\mathrm{C}([0,1],A^m)\to A^m\oplus A^m$ given by evaluation at vertices (hence surjective) and $\mathrm{C}(V,A)\to A^m\oplus A^m$  suitably defined (see, e.g. \cite[Section 3]{APS}). It is clear that the two entries in the bottom of the diagram have stable rank one. By Theorem \ref{StableRankOfInterval} the entry in the upper right corner has stable rank one. It then follows from \cite[Theorem 4.1 (i)]{bp} that $\mathrm{sr}(\mathrm{C}(X,A))=1$. 

Finally, if $X$ is one dimensional, we may write $X$ as a (countable) inverse limit of finite graphs, and so $\mathrm{C}(X,A)$ is an inductive limit of algebras that have stable rank one. Therefore the stable rank of $\mathrm{C}(X,A)$ is also one.
\end{proof}

To prove the corresponding result to Theorem \ref{StableRankOfInterval} for more general spaces of dimension one, we need to generalize Proposition \ref{NecessaryConditions}. The proof is inspired by \cite[Proposition 5.2]{nop}. We thank Hannes Thiel for providing us with an argument that allows us to drop unnecessary assumptions in an earlier draft.

Recall that, given a compact metric space $X$, a continuous map $f\colon X\to [0,1]$ is \emph{essential} if whenever a continuous map $g\colon X\to [0,1]$ that agrees with $f$ on $f^{-1}(\{0,1\})$ must be surjective. A classical result of Alexandroff shows that if $X$ is one-dimensional space, then there is an essential map from $X$ to $[0,1]$. (A suitable generalization of the above definition can be used to characterize when a space has dimension $\geq n$, see \cite{Eng}.)

\begin{proposition}
Let $A$ be any C$^*$-algebra and $X$ be a compact metric space with $\dim(X)=1$. If $\mathrm{sr}(\mathrm{C}(X,A))=1$, then $A$ has no $\mathrm{K}_1$-obstructions.
\end{proposition}

\begin{proof}
Since $A$ is a quotient of $\mathrm{C}(X,A)$ it is clear that the stable rank of $A$ must be one. Further, arguing as in the proof of Proposition \ref{NecessaryConditions}, it suffices to show that $\mathrm{K}_1(A)=0$.

Suppose, to reach a contradiction, that there is a unitary $u$ in $A$ not connected to $1$. Let $\mathrm{d}$ be the metric that induces the topology on $X$. Since $X$ is one-dimensional, there is an essential map $f\colon X\to [0,1]$. Let $S=f^{-1}(\{0\})$ and $T=f^{-1}(\{1\})$, which are disjoint closed sets and hence $\mathrm{d}(S,T)>0$. We may assume that $\mathrm{d}(S,T)=1$. Now define a continuous function $v\colon X\to A$ as follows:
\[
v(x)=(1-\mathrm{d}(x,T))_+\cdot u+(1-\mathrm{d}(x,S))_+\cdot 1\,.
\]
Notice that, by definition, $v_{|S}=1$ and $v_{|T}=u$. As $\mathrm{C}(X,A)$ has stable rank one, there is a map $w\colon X\to A^{-1}$ such that $||v-w||<1$. Denote by $A^{-1}_0$ the connected component of $A^{-1}$ containing the identity. We have that $S\subseteq w^{-1}(A^{-1}_0)$ and $T\subseteq w^{-1}(A^{-1}\setminus A^{-1}_0)$ as $u\notin A^{-1}_0$ by assumption. Note that $A^{-1}_0$ is both open and closed in $A^{-1}$, so by continuity of $w$ we obtain that $S':=w^{-1}(A^{-1}_0)$ and $T':=w^{-1}(A^{-1}\setminus A^{-1}_0)$ form a partition of $X$ consisting of clopen sets. Thus we can define a (non-surjective) continuous function $h\colon X\to [0,1]$ such that $h(S')=0$ and $h(T')=1$, and this contradicts the essentiality of $f$.

\end{proof}

We collected everything for a repetition of the arguments of the proof of Theorem \ref{StableRankOfInterval} in a more general setting.

\begin{theorem}\label{StableRankOfDimOne}
Let $A$ be any C$^*$-algebra with $\mathrm{sr}(A)=1$ and $X$ be a  compact metric space of dimension one. Then 
$$\mathrm{sr}(\mathrm{C}(X,A))=\left \{ \begin{array}{ll} 1,&\mbox{ if $A$ has no $\mathrm{K}_1$-obstructions}\\ 2,& \text{else.}\end{array}\right.$$ 
\end{theorem}

Although we won't need it in the following, we would like to point out that Theorem \ref{StableRankOfDimOne} determines the real rank of certain algebras by an application of the well-known  inequality stating that $RR(A)\leq 2\mathrm{sr}(A)-1$ and \cite[Proposition 5.1]{nop}.

\begin{corollary}
Let $A$ be C$^*$-algebra with no $\mathrm{K}_1$-obstructions, and let $X$ be a compact metric space of dimension one. Then $RR(\mathrm{C}(X,A))=1$.
\end{corollary}
 
In view of Theorem \ref{StableRankOfDimOne}, it is natural to ask if the same can be obtained for continuous fields over $X$.  We recall the main definitions (see e.g. \cite{nilsen,Dixmier,dadarlat}). If $X$ is a compact Hausdorff space, a \emph{$\mathrm{C}(X)$-algebra} is a C$^*$-algebra together with a unital $^*$-homomorphism $\mathrm{C}(X)\to Z(\mathcal {M}(A))$. If $Y\subseteq X$ is a closed set, let $A(Y)=A/\mathrm{C}_0(X\setminus Y)A$, which also becomes a $\mathrm{C}(X)$-algebra. Denote by $\pi_Y\colon A\to A(Y)$ the natural quotient map. In the case that $Y=\{x\}$, we then write $A_x$ and $\pi_x$. The algebra $A_x$ is referred to as the \emph{fiber} of $A$ at $x$.

For a $\mathrm{C} (X)$-algebra $A$, the map $x\mapsto ||a(x)||$ is upper semicontinuous, and if it is continuous, we then say that $A$ is a \emph{continuous field}.

The natural question is then whether a continuous field of C$^*$-algebras $A$ over a one-dimensional space $X$, all of whose fibers have no $\mathrm{K}_1$-obstructions, is necessarily of stable rank one. And, conversely, if $\mathrm{sr}(A)=1$ for a continuous field $A$ over a one-dimensional space $X$ implies $\mathrm{K}_1(A_x)=0$ for all $x\in X$. We attain a positive answer to the first named question, but the second question can be answered in the negative, even for $X=[0,1]$ as we show below. 

\begin{theorem}\label{StableRankOfCtsFields}
Let $X$ be a one-dimensional, compact metric space, and let $A$ be a continuous field over $X$ such that each fiber $A_x$ has no $\mathrm{K}_1$-obstructions. Then $\mathrm{sr}(A)=1$.
\end{theorem}
\begin{proof}
As $X$ is metrizable and one-dimensional, we can apply \cite[Theorem 1.2]{ngsudo} to obtain that $\mathrm{sr}(A)\leq \sup_{x\in X} \mathrm{sr}(\mathrm{C}([0,1],A_x))$. Now the result follows immediately from Theorem \ref{StableRankOfDimOne}.
\end{proof}

The previous result yields:

\begin{corollary}{\rm (cf. \cite[Lemma 3.3]{dadarlatelliottniu})}
Let $X$ be a one-dimensional, compact metric space, and let $A$ be a continuous field of AF algebras. Then $\mathrm{sr}(A)=1$.
\end{corollary}

\begin{corollary}
Let $X$ be a one-dimensional, compact metric space, and let $A$ be a continuous field of simple AI algebras. Then $\mathrm{sr}(A)=1$.
\end{corollary}

If $A$ is a locally trivial field of C$^*$-algebras with base space the unit interval then it is clear by the methods above that $\mathrm{sr}(A)=1$ implies that $\mathrm{K}_1(A_x)$ must be trivial for all $x$. For general continuous fields, this implication is false.

\begin{proposition}\label{lem:counterexample}
 Let $B\subset C$ be C$^*$-algebras with stable rank one such that $C$ has no $\mathrm{K}_1$ obstructions. Let 
\[A=\{f\in\mathrm{C}([0,1],C)\mid f(0)\in B\}\,.\] Then $A$ is a continuous field over $[0,1]$ with stable rank one.
\end{proposition}

 \begin{proof}
It is clear that $A$ is a $C([0,1])$-algebra, which is moreover a continuous field.

Observe that $A$ can be obtained as the pullback of the diagram
\[\xymatrix{
A\ar@{-->}[r]\ar@{-->}[d] &  B\ar@{^{(}->}[d]^i \\ \mathrm{C}([0,1],C)\ar@{->>}[r]^>>>>>>{\mathrm{ev}_0} & C}
\]where ev$_0$ is the map given by evaluation at 0. Since the rows are surjective, we have by \cite[Theorem 4.1]{bp} that $\mathrm{sr}(A)\leq \text{max}\{\mathrm{sr}(B),\mathrm{sr}(\mathrm{C}([0,1],C))\}$.
As $C$ has no $\mathrm{K}_1$ obstructions and $\mathrm{sr}(B)=1$, we have $\mathrm{sr}(A)=1$ by Theorem \ref{StableRankOfInterval}. 
 \end{proof}

\begin{proposition}
There exists a (nowhere trivial) continuous field $A$ over $[0,1]$ such that $\mathrm{sr}(A)=1$ and $\mathrm{K}_1(A_x)\neq 0$ for a dense subset of $[0,1]$.
\end{proposition}

\begin{proof}
 Let $C=\mathrm{C}(X)$, and $B=\mathrm{C}(\mathbb T)$ where $X$ denotes the cantor set and $\mathbb T$ the unit circle. There exists a surjective map $\pi\colon X \to \mathbb T$ and hence there is an embedding 
$i\colon B \to C$.  Choose a dense sequence $\{x_n\}_n\subset [0,1]$ and define 
\[ C_n:=\{f\in \mathrm{C}([0,1],C)\mid f(x_n)\in i(B)\}\,.\]
Since $X$ is zero dimensional, $C$ is an AF-algebra and hence has no $\mathrm{K}_1$ obstructions. Therefore $C_n$ is a continuous field over $[0,1]$ of stable rank one by Proposition \ref{lem:counterexample}.
Note that $C_n(x_n)\cong B$ which has non trivial $\mathrm{K}_1$. We now proceed as in the proof of \cite[Corollary 8.3]{DadarlatElliott} to obtain a dense subset of such singularities. 

Let $A_1=C_1$, $A_{n+1}=A_n\otimes_{\mathrm{C}[0,1]} C_{n+1}$ and $A=\varinjlim (A_n,\theta_n)$ where $\theta_n(a)=a\otimes 1$ (see \cite{blanchard}). Note that $A_n$ can be described as 
\[A_n=\{f\in \mathrm{C}([0,1],C^{\otimes n})\mid f(x_i)\in C^{\otimes i-1}\otimes i(B) \otimes C^{\otimes n-i}, i=1,\dots, n\}\,,\]
and now $\theta_n(f)(x)=f(x)\otimes 1$. Hence $A_n$ is clearly a continuous field which can moreover be described by the following pullback diagram 
\[\xymatrix{
 A_n\ar@{-->}[r]\ar@{-->}[d] &  B\otimes C^{\otimes n-1}\oplus \cdots \oplus C^{\otimes n-1}\otimes B \ar@{^{(}->}[d] \\ \mathrm{C}([0,1],C^{\otimes n})\ar@{->>}[r]^>>>>>>>>>{\mathrm{ev}_{x_1,\dots,x_n}} & C^{\otimes n}\oplus \stackrel{n}{\dots} \oplus C^{\otimes n}}
 \]
Again, since $C^{\otimes n}$ is an AF algebra it has no $K_1$ obstructions. Then, a similar argument as that in the proof of Proposition \ref{lem:counterexample} applies to conclude that $A_n$ has stable rank one. Moreover, $A$ has stable rank one since it is an inductive limit of stable rank one algebras $A_n$ (and is moreover commutative).

Now, for any $x\in [0,1]$, the fiber $A(x)$ can be computed as $\varinjlim A_n(x)$. Hence, if $x\not\in\{x_n\}_n$, $A(x)\cong\varinjlim C^{\otimes n}\cong \varinjlim \mathrm{C}(X^n)\cong \mathrm{C}(\varprojlim X^n)$. Since $\varprojlim X^n$ is also zero dimensional,
$A(x)$ is an AF-algebra and thus has trivial $K_1$. 

Assume $x=x_k\in \{x_n\}_n$. Now for any $n\geq k$, 
\[
A_n(x_k)\cong C^{\otimes k-1}\otimes B \otimes C^{\otimes n-k}\cong \mathrm{C}(X^{k-1}\times \mathbb T \times X^{n-k})\,.
\] 
An application of the Künneth formula shows that $K_1(A_n(x_k))\cong \mathrm{C}(X^{n-1},\mathbb Z)$ and thus $K_1(A(x_k))\cong \varinjlim \mathrm{C}(X^{n-1},\mathbb Z)\cong \mathrm{C}(\prod_{i=1}^\infty X,\mathbb Z)\neq 0$.
\end{proof}

\section{Hereditariness}\label{SectionHer}

In this short section we study the hereditary character of certain continuous fields. This will be used in the sequel as, since mentioned earlier, in this setting the classical and the stabilized Cuntz semigroup carry the same information. We start by recalling the definitions.

Let $A$ be a C$^\ast$-algebra, and $a,b\in A_+$. We say that $a$ is \emph{Cuntz-subequivalent} to $b$, in symbols $a\preceq b$, if there is a sequence $(v_n)$ in $A$
such that $a=\lim_n v_nbv^*_n$. We say that $a$ is \emph{Cuntz-equivalent} to $b$, and we write $a\sim b$, if both conditions $a\preceq b$ and $b\preceq a$ are satisfied. Upon extending this relation
to $M_\infty (A)_+$, one obtains an ordered set $\mathrm{W}(A)=M_\infty (A)_+ /\!\!\sim$. We denote the equivalence class of $a\in M_{\infty}(A)_+$ by $\la a\ra$, and then the above set becomes a partially ordered abelian semigroup when it is equipped with the operation  $\la a\ra +\la b\ra=\la \left(\begin{smallmatrix} a&0\\ 0&b \end{smallmatrix}\right)\ra=\la a\oplus b\ra$, and order given by $\la a\ra\leq \la b\ra$ if $a\preceq b$. The semigroup $\mathrm{W}(A)$ is referred to as the \emph{Cuntz semigroup}.

In \cite{CEI}, Coward, Elliott and Ivanescu introduced a category of partially ordered semigroups $\Cu$, to which the 
Cuntz semigroup of a stable C$^\ast$-algebra belongs. Furthermore, they proved that 
$\Cu(A):=\mathrm{W}(A\otimes \mathcal K)$ defines a sequentially continuous functor from the category of C$^*$-algebras to $\Cu$. The semigroup $\Cu(A)$ is sometimes called the \emph{stabilized Cuntz semigroup} in order to distinguish it from $\mathrm{W}(A)$. 

Semigroups in the category $\Cu$ have a rich ordered structure not always present in $\mathrm{W}(A)$. Hence, and in order to fit $\mathrm{W}(A)$ into this categorical description, a new category called $\PreCu$ was introduced in \cite{ABP}, where $\mathrm{W}(A)$ belongs in a number of instances. It is shown in \cite[Proposition 4.1]{ABP} that there is a functor from 
$\PreCu$ to $\Cu$ which is left-adjoint to the identity. This functor is basically a completion of semigroups
and, for a wide class of C$^*$-algebras, it sends $\mathrm{W}(A)$ to $\Cu(A)$. We recall some of the main facts below.

Recall that, for a partially ordered semigroup $M$ and elements $a$, $b\in M$, we say that $a$ is \emph{compactly contained} in $b$, in symbols $a\ll b$,  if for any increasing
 sequence $(b_n)$ in $M$ such that $\sup(b_n)$ exists and
$b\leq \sup(b_n)$ there exists $n_0$ such that $a\leq b_{n_0}$. When an increasing sequence $(b_n)$ satifies that $b_n\ll b_{n+1}$, then we say that $(b_n)$ is a \emph{rapidly increasing} sequence.

\begin{definition}[\cite{CEI, ABP}]
 Let $\PreCu$ be the category whose objects are those partially ordered abelian semigroups $M$ satisfying the following properties:
\begin{enumerate}[\rm(i)]
 \item Every element in $M$ is the supremum of a rapidly increasing sequence.
 \item The relation $\ll$ and suprema are compatible with addition.
\end{enumerate}
Maps of PreCu are semigroup maps preserving suprema of increasing sequences (when they exist), and the relation $\ll$.

In this light, $\Cu$ may be defined as the full subcategory of $\PreCu$ whose objects are those partially ordered abelian semigroups (in $\PreCu$) for which every increasing
sequence has a supremum.
\end{definition}

 Given a semigroup $M$ in $\PreCu$, we say that a pair $(N,\iota)$ is a \emph{completion} of $M$ if
\begin{enumerate}[\rm(i)]
 \item $N$ is an object of $\Cu$,
\item $\iota\colon M\to N$ is an order-embedding in $\PreCu$, and
\item for any $x\in N$, there is a rapidly increasing sequence $(x_n)$ in $M$ such that $x=\sup \iota(x_n)$.
\end{enumerate}

It was shown in \cite[Theorem 5.1]{ABP} that, for $M\in\PreCu$, there exists a (unique) object $\overline{M}$ in $\Cu$ and an order-embedding $\iota\colon M\to \overline{M}$ in
$\PreCu$, satisfying that $(\overline{M},\iota)$ is the completion of $M$. 

If  $M$ and $N$ are partially ordered semigroups, an order-embedding $\iota\colon M\to N$ is called \emph{hereditary} if, whenever $x\in N$ and $y\in \iota(M)$
satisfy $x\leq y$, then $x\in \iota(M)$. If the order-embedding $\iota\colon \mathrm{W}(A)\to \mathrm{W}(A\otimes\mathcal K)=\Cu(A)$ is hereditary, then we will say that $\mathrm{W}(A)$ is hereditary. In this case, $\mathrm{W}(A)\in\PreCu$ and its completion is $(\mathrm{W}(A\otimes\mathcal K),\iota)$ (see \cite[Theorem 6.1]{ABP}). There are no examples known of C$^*$-algebras $A$ for which $\mathrm{W}(A)$ is not hereditary.

Recall (\cite{BRTTW}) that if $A$ is a unital C$^*$-algebra, the \emph{radius of comparison} of $(\Cu(A),[1_A])$, denoted by $r_A$, is defined as the infimum of $r\geq 0$ satisfying that if $x,y\in \Cu(A)$ are such that $(n+1)x+m[1_A]\leq ny$ for some $n,m$ with $\frac{m}{n}>r$, then $x\leq y$. In general, $r_A\leq \mathrm{rc}(A)$, where $\mathrm{rc}(A)$ is the radius of comparison of the algebra (see, e.g. \cite{toms:plms}) and equality holds if $A$ is residually stably finite (i.e. all quotients of $A$ are stably finite) (see Proposition 3.2.3 in \cite{BRTTW}). It is known that, if $A$ has stable rank one or finite radius of comparison, then $\mathrm{W}(A)$ is hereditary (\cite{ABP}, \cite{BRTTW}). In particular, this holds if $A$ is a continuous field over a one dimensional, compact metric space such that each fiber has no $\mathrm{K}_1$-obstructions by Theorem \ref{StableRankOfCtsFields}.

The following is probably well known. We include a proof for completeness.
\begin{lemma}
\label{lem:stfinite}
Let $X$ be a compact Hausdorff space, and let $A$ be a $\mathrm{C}(X)$-algebra such that $A_x$ has stable rank one for all $x$. Then $A$ is residually stably finite.
\end{lemma}
\begin{proof}
Let $I$ be an ideal of $A$, which is also a $\mathrm{C}(X)$-algebra, as well as is the quotient $A/I$, with fibers $(A/I)_x\cong A/(\mathrm{C}_0(X\setminus\{x\})A+I)$. As these are quotients of $A_x$, they have stable rank one, so in particular they are stably finite, and this clearly implies $A/I$ is stably finite.
\end{proof}

\begin{proposition}
\label{prop:ctsfieldhered} Let $X$ be a finite dimensional compact Hausdorff space, and let $A$ be a continuous field over $X$ whose fibers are simple, finite, and $\mathcal Z$-stable. Then $\mathrm{W}(A)$ is hereditary. 
\end{proposition}
\begin{proof}
We know from Lemma \ref{lem:stfinite} that $A$ is residually stably finite, so $\mathrm{rc}(A)=\mathrm{r}_A$. We also know from \cite[Theorem 4.6]{hrw} that $A$ itself is $\mathcal Z$-stable, whence $\Cu(A)$ is almost unperforated (\cite[Theorem 4.5]{Rorijm}). Thus $r_A=0$. This implies that $A$ has radius of comparison zero and then \cite[Theorem 4.4.1]{BRTTW} applies to conclude that $\mathrm{W}(A)$ is hereditary.
\end{proof}

\begin{remark}\label{rem:hereditary}{\rm
 In the previous proposition, finite dimensionality is needed to ensure $\mathcal{Z}$-stability of the continuous field. Notice that in the case of a trivial continuous field $A=C(X,D)$ where $D$ is simple, finite, and $\mathcal Z$-stable, the same argument can be applied for arbitrary (infinite dimensional) compact Hausdorff spaces.}
\end{remark}

\begin{definition}
 Let $X$ be a topological space, let $M$ be a semigroup in $\PreCu$, and let $f\colon X\to M$ be a map. We say that $f$ is \emph{lower semicontinuous} if, for all
$a\in M$, the set $\{t\in X\mid a\ll f(t)\}$ is open in $X$. We shall denote the set of lower semicontinuous functions by $\mathrm{\mathrm{Lsc}}(X,M)$ and the set of bounded lower semicontinuous functions by $\mathrm{\mathrm{Lsc}}_{\mathrm{b}}(X,M)$. Note that, if $M\in \Cu$, then $\mathrm{Lsc}(X,M)=\mathrm{Lsc}_{\mathrm{b}}(X,M)$. Furthermore, the sets just defined become ordered semigroups when equipped with pointwise order and addition.
\end{definition}

Recall that a compact metric space $X$ is termed \emph{arc-like} provided $X$ can be written as the inverse limit of intervals. Note that arc-like spaces include non-trivial examples, such as the pseudo-arc which is a one dimensional space that does not contain an arc (see e.g. \cite{nadler}).

\begin{proposition}
\label{lem:cuntzse}
Let $X$ be an arc-like compact metric space, and let $A$ be a unital, simple C$^*$-algebra with stable rank one, and finite radius of comparison. Then $\mathrm{W}(\mathrm{C}(X,A))$ is hereditary.
\end{proposition}
\begin{proof}
We will prove that $\mathrm{C}(X,A)$ has finite radius of comparison, and then appeal to \cite[Theorem 4.4.1]{BRTTW}. 
Since $\Cu$ is a continuous functor and $X$ is an inverse limit of intervals, we can combine  \cite[Theorem 2.6]{adps} and \cite[Proposition 5.18]{APS} to obtain $\Cu(\mathrm{C}(X,A))\cong \mathrm{Lsc}(X,\Cu(A))$.
Now, by Lemma \ref{lem:stfinite}, $\mathrm{C}(X,A)$ is residually stably finite, and hence by \cite[Proposition 3.3]{BRTTW} $\mathrm{rc}(\mathrm{C}(X,A))=r_{\mathrm{C}(X,A)}$. Since the order in $\mathrm{Lsc}(X, \Cu(A))$ is the pointwise order, it is easy to verify that $r_{\mathrm{C}(X,A)}\leq r_A$. Note that this is in fact an equality as $A$ is a quotient of $\mathrm{C}(X,A)$ (see condition (i) in \cite[Proposition 3.2.4]{BRTTW}). 
\end{proof}

\section{Lower semicontinuous functions, continuous sections, and Riesz interpolation}
\label{sec:lsc}

In this section we prove that the Grothendiek group of the Cuntz semigroup of certain continuous fields has Riesz interpolation. In some cases we apply the results on hereditariness from Section \ref{SectionHer} together with results in \cite{APS}, \cite{adps} and  \cite{ABP2}, and then if $A$ is simple, unital, ASH, with slow dimension growth, we apply the description of $\mathrm{W}(\mathrm{C}(X,A))$ given in \cite{Tiku}.

We recall the necessary definitions.

Let $(M,\leq)$ be a partially ordered semigroup. We say that $M$ is an \emph{interpolation semigroup} if  it satisfies the \emph{Riesz interpolation property}, that is,  whenever $a_1,a_2,b_1,b_2\in M$ are such that $a_i\leq b_j$ for
 $i,j=1,2$, there exists $c\in M$ such that $a_i\leq c\leq b_j$ for $i,j=1,2$. If the order is algebraic and $M$ is cancellative, then this property is well known to be equivalent to the Riesz decomposition property and also to the Riesz refinement property (see, e.g. \cite{Goo}).

Define $\mathcal{C}$ as the full subcategory of $\PreCu$ whose objects are those semigroups $M$ such that $\iota\colon M\to \overline M$ is hereditary.

\begin{lemma}\label{interpolation}
 Let $M$ be a semigroup in $\mathcal{C}$. Then $M$ is an interpolation semigroup if and only if its completion $\overline{M}$ is.
\end{lemma}
\begin{proof}
Assume that $M$ satisfies the Riesz interpolation property and let $a_i\leq b_j$ be elements in $\overline{M}$ for $i,j\in\{1,2\}$. Denote by $\iota\colon M\to\overline{M}$
the corresponding order-embedding completion map.  We may write $a_i=\sup(\iota(a^n_i))$ and $b_j=\sup(\iota(b^n_j))$ for
$i,j\in \{1,2\}$, where $(a^n_i)$ and $(b^n_j)$ are rapidly increasing sequences in $M$. Find $m_1\geq 1$ such that $\iota(a_i^1)\leq\iota(b_j^ {m_1})$. Then $a_i^1\leq b_j^{m_1}$ and by the Riesz interpolation property there is $c_1\in M$ such that 
$a_i^1\leq c_1\leq b_j^{m_1}$. Suppose we have constructed $c_1\leq\dots\leq c_n$ in $M$ and $m_1<\dots <m_n$ such that $a_i^k\leq c_k\leq b_j^{m_k}$ for each $k$. Find $m_{n+1}>m_n$ such that $a_i^{n+1},c_n\leq b_j^{m_{n+1}}$, and by the interpolation property there exists $c_{n+1}\in M$ with $a_i^{n+1},c_n\leq c_{n+1}\leq  b_j^{m_{n+1}}$. Now let $\bar{c}=\sup\iota(c_n)\in\overline M$, and it is clear that $a_i\leq\bar{c}\leq b_j$ for all $i,j$.

Since $\iota$ is a hereditary order-embedding, the converse implication is immediate.
\end{proof}

\begin{lemma}
\label{lemaprevi}
Let $M\in\Cu$ satisfy the property that, whenever $a_i, b_j$ ($i,j=1,2$) are elements in $M$ and $a_i\ll b_j$ for all $i$ and $j$, then, for every $a_i'\ll a_i$, there is $c\in M$ such that $a_i'\ll c\ll b_j$. Then $M$ is an interpolation semigroup.
\end{lemma}
   \begin{proof}
Suppose that $a_i\leq b_j$ in $M$ (for $i,j=1,2$). Write $a_i=\sup a_i^n$ and $b_j=\sup b_j^m$, where $(a_i^n)$ and $(b_j^m)$ are rapidly increasing sequences in $M$. Since $a_i^1\ll a_i^2\ll b_j$, there is $m_1\geq 1$ such that $a_i^2\ll b_j^{m_1}$. By assumption, there are elements $c_1\ll c_1'$ in $M$ such that $a_i^1\ll c_1\ll c_1'\ll b_j^{m_1}$. Now, there is $m_2>m_1$ such that $c_1', {a_i^3}\ll b_j^{m_2}$, so a second application of the hypothesis yields elements $c_2\ll c_2'$ with $c_1,a_i^2\ll c_2\ll c_2'\ll b_j^{m_2}$. Continuing in this way we find an increasing sequence $(c_n)$ in $M$ whose supremum $c$ satisfies $a_i\leq c\leq b_j$.    
   \end{proof}

\begin{remark}
\label{remarca}
{\rm
 It is proved in \cite[Theorem 5.15]{APS} that, if $X$ is a finite dimensional, compact metric space and $M\in \Cu$ is countably based, then $\mathrm{Lsc}(X,M)$ is also in $\Cu$. As it turns out from the proof of this fact, every function in $\mathrm{Lsc}(X,M)$ is a supremum of a rapidly increasing sequence of functions, each of which takes finitely many values.}
\end{remark}
\begin{proposition}\label{tres}
 Let $M\in \mathcal{C}$ be countably based, let $(\overline{M},\iota)$ be its completion, and let $X$ be a finite dimensional, compact metric space. Then 
$\mathrm{Lsc}_{\mathrm{b}}(X,M)$ is an object of $\mathcal{C}$ and $(\mathrm{Lsc}(X,\overline{M}),i)$ is its completion, where $i$ is induced by $\iota$.
\end{proposition}

\begin{proof}
Notice that $\mathrm{Lsc}(X,\overline{M})\in\Cu$ and that $i(f)=\iota\circ f$ defines an order-embedding. 

Given $f\in \mathrm{Lsc}(X,\overline{M})$, write $f=\sup f_n$, where $(f_n)$ is a rapidly increasing sequence of functions taking finitely many values. Since $f_n\ll f$ and thus $f_n(x)\ll f(x)$ for every $x\in X$, the range of $f_n$ is a (finite) subset of $\iota(M)$. Therefore each $f_n$ belongs to $\mathrm{Lsc}_{\mathrm{b}}(X,M)$.
\end{proof}

\begin{proposition}\label{inter}
Let $M$ be a countably based, interpolation semigroup in $\Cu$, and let $X$ be a finite dimensional, compact metric space. Then $\mathrm{Lsc}(X,M)$ is an interpolation semigroup.
  \end{proposition}
\begin{proof}
We apply Lemma \ref{lemaprevi}, so assume $f_i\ll f_i'\ll g_j$ for $i,j=1,2$. Since $f_i\ll f_i'$, given $x\in X$ there is a neighborhood $U_x'$ of $x$ and $c_{i,x}\in M$ such that $f_i(y)\ll c_{i,x}\ll f_i'(y)$ for all $y\in U_x'$ (by \cite[Proposition 5.5]{APS}). Now $c_{i,x}\ll f_i'(y)\ll g_j(y)$ for each $y\in U_x'$, so in particular it will hold for $x$. Since $M$ is an interpolation semigroup, there is $d_x\in M$ such that $c_{i,x}\ll d_x\ll g_j(x)$ and, by lower semicontinuity of $g_j$, there is a neighborhood $U_x''$ such that $d_x\ll g_j(y)$ for every $y\in U_x''$. Thus, if $U_x=U_x'\cap U_x''$, we have $f_i(y)\ll d_x\ll g_j(y)$ for all $y\in U_x$.

We may now run the argument in \cite[Proposition 5.13]{APS} to patch the values $d_x$ into a function $h\in \mathrm{Lsc}(X,M)$ that takes finitely many values and $f_i\ll h\ll g_j$, as desired.
\end{proof}
\begin{corollary}\label{interC}
 Let $M$ be a countably based semigroup in $\mathcal{C}$ and let $X$ be a finite dimensional, compact metric space. Then, if $M$ is an interpolation semigroup, so is $\mathrm{Lsc}_{\mathrm{b}}(X,M)$.
\end{corollary}
\begin{proof}
By Lemma \ref{interpolation} followed by Proposition \ref{inter}, the semigroup $\mathrm{Lsc}(X,\overline M)$ is an interpolation semigroup, where $\overline M$ is the completion of $M$. On the other hand, by Proposition \ref{tres}, $\mathrm{Lsc}(X,\overline M)$ is the completion of $\mathrm{Lsc}_{\mathrm{b}}(X,M)\in\mathcal C$, whence another application of Lemma \ref{interpolation} yields the conclusion.
\end{proof}

Let $(M,\leq)$ be a partially ordered semigroup. We denote by $\mathrm{G}(M)$ its Grothendieck group, and order $\mathrm{G}(M)$ by setting $\mathrm{G}(M)^+=\{[a]-[b]\mid b\leq a\}$ as its positive cone. This defines a partial order on $\mathrm{G}(M)$ and, for $a,b,c,d\in M $
$$[a]-[b]\leq [c]-[d]\text{ in }\mathrm{G}(M)\iff a+d+e\leq b+c+e \text{ in } M \text { for some }e\in M.$$

If $A$ is a C$^*$-algebra, we denote by $\mathrm{K}_0^*(A)$ the Grothendieck group of $\mathrm{W}(A)$ and by
$[a]-[b]$ the elements of this group, where
 $a,b\in M_\infty(A)_+$ (see \cite{Cu}). It is easy to see that the set of states on a semigroup 
can be naturally identified with the set of states on its Grothendieck group. 

Condition (i) in the result below is a special case of Theorem \ref{th:ctsfields}, but the proof in this case is easier.

\begin{theorem}
\label{interpolacio}
Let $X$ be a compact metric space, and let $A$ be a separable, C$^*$-algebra of stable rank one. Then $\mathrm{K}_0^*(\mathrm{C}(X,A))$ is an interpolation group in the following cases:
\begin{enumerate}[{\rm (i)}]
\item $\dim X\leq 1$, $\mathrm{K}_1(A)=0$ and has either real rank zero or is simple and $\mathcal Z$-stable. 
\item $X$ is arc-like, $A$ is simple and either has real rank zero and finite radius of comparison, or else is $\mathcal Z$-stable.
\item $\dim X\leq 2$ with vanishing second \v Cech cohomology group $\check{\mathrm{H}}^2(X,\mathbb Z)$, and $A$ is an infinite dimensional AF-algebra. 
\end{enumerate}
\end{theorem}
\begin{proof}
(i): If $A$ has real rank zero, it was proved in \cite[Theorem 2.13]{Per} that $\mathrm{W}(A)$ satisfies the Riesz interpolation property, and then so does $\Cu(A)$ by Lemma \ref{interpolation}. In the case that $A$ is simple and $\mathcal Z$-stable, $\Cu(A)$ is an interpolation semigroup by \cite[Proposition 5.4]{Tiku}. Since, by \cite[Theorem 3.4]{APS}, $\Cu(\mathrm{C}(X,A))$ is order-isomorphic to $\mathrm{Lsc}(X,\Cu(A))$ we obtain, using Proposition \ref{inter}, that $\Cu(\mathrm{C}(X,A))$ is an interpolation semigroup in both cases. By Corollary \ref{CorDimOne}, $\mathrm{C}(X,A)$ has stable rank one, and so $\mathrm{W}(\mathrm{C}(X,A))$ is hereditary, hence also an interpolation semigroup by Lemma \ref{interpolation}. Thus $\mathrm{K}_0^*(\mathrm{C}(X,A))$ is an interpolation group (using \cite[Lemma 4.2]{Per}).

(ii): By Proposition \ref{lem:cuntzse} and its proof we see that $\mathrm{W}(\mathrm{C}(X,A))$ is hereditary and that $\Cu(\mathrm{C}(X,A))$ is order-iso\-mor\-phic to $\mathrm{Lsc}(X,\Cu(A))$. Now the proof follows the lines of the previous case.

(iii): This follows as above, using \cite[Corollary 3.6]{APS}, so that $\Cu(\mathrm{C}(X,A))$ is order-iso\-mor\-phic to $\mathrm{Lsc}(X,\Cu(A))$, and the proof of Proposition \ref{lem:cuntzse}, so that $\mathrm{W}(\mathrm{C}(X,A))$ is hereditary.
\end{proof}

\begin{remark}
\upshape{
Note that if $\mathrm{K}_1(A)\neq 0$ in case (ii) of Theorem \ref{interpolacio}, then $\mathrm{sr}(\mathrm{C}([0,1],A))= 2$ (see \cite[Proposition 5.2]{nop} and Section \ref{sr}). Notice also that it remains an open problem to decide whether a simple C$^*$-algebra $A$ with real rank zero and stable rank one must have weakly unperforated $\mathrm{K}_0(A)$ (which is in fact equivalent to the fact that $A$ has strict comparison).}
\end{remark}

We now turn our consideration to algebras of the form $\mathrm{C}(X,A)$ where $A$ is a unital, simple, non-type I ASH-algebra with slow dimension growth. In this setting we are able to obtain the same conclusion as above without the necessity to go over proving interpolation of $\Cu(\mathrm{C}(X,A))$. We first need a preliminary result.

\begin{proposition}\label{Grot}
 Let $N$ be a partially ordered abelian semigroup and let $M$ be an ordered subsemigroup of $N$ such that $M+N\subseteq M$. Then $\mathrm{G}(M)$ and $\mathrm{G}(N)$ are isomorphic as partially ordered abelian groups.
\end{proposition}
\begin{proof}
 Let us denote by $\gamma\colon M\to \mathrm{G}(M)$ and $\eta\colon N\to \mathrm{G}(N)$ the natural Grothendieck maps. Fix
$c\in M$, and define $\alpha\colon N\to \mathrm{G}(M)$ by $\alpha(a):= \gamma(a+c)-\gamma(c)$.
Using that $M+N\subseteq M$, it is easy to verify that the definition of $\alpha$ does not depend on $c$. Now, if $a$, $b\in N$, we have
\begin{align*}
\alpha(a+b) & =\gamma(a+b+c)-\gamma(c) =\gamma(a+b+c)+\gamma(c)-2\gamma(c) \\&=\gamma(a+c+b+c)-2\gamma(c) =(\gamma(a+c)-\gamma(c))+(\gamma(b+c)-\gamma(c))\\ &=\alpha(a)+\alpha(b)\,,
\end{align*}
so that $\alpha$ is a homomorphism. It is clear that $\alpha(N)\subseteq \mathrm{G}(M)^+$. 

By the universal property of the Grothendieck group, there exists a group homomorphism $\alpha'\colon \mathrm{G}(N)\to \mathrm{G}(M)$ such that $\alpha'(\eta(a)-\eta(b))=\alpha(a)-\alpha(b) $. Note that $\alpha'$ is injective. Indeed, if $\alpha(a)-\alpha(b)=0$, then $\gamma(a+c)=\gamma(b+c)$ and so $a+c+c'=b+c+c'$ for some $c'\in M$, and thus $\eta(a)=\eta(b)$.

 If $\eta(a)-\eta(b)	\in \mathrm{G}(N)^+$ with $b\leq a$ in $N$, then $b+c\leq a+c$ in $M$ and so $\gamma(b+c)-\gamma(a+c)\in \mathrm{G}(M)^+$. Therefore 
\begin{align*}
\alpha'(\eta(a)-\eta(b)) & =\alpha(a)-\alpha(b)\\ & =\gamma(a+c)-\gamma(c)-(\gamma(b+c)-\gamma(c))=\gamma(a+c)-\gamma(b+c)\,,
\end{align*}
which shows that $\alpha' (\mathrm{G}(N)^+)\subseteq \mathrm{G}(M)^+$.

Observe that, if $a\in M\subseteq N$, then $\alpha(a)=\gamma(a+c)-\gamma(c)=\gamma(a)$. This implies that any element in $\mathrm{G}(M)$ has the form 
\[
\gamma(a)-\gamma(b)=\gamma(a+c)-\gamma(b+c)=\alpha'(\eta(a+c)-\eta(b+c))\,
\]
 and so $\alpha'$ is surjective and $\alpha'(\mathrm{G}(N)^+)=\alpha(\mathrm{G}(M)^+)$.

\end{proof}

Given semigroups $N$ and $M$ as above, we will say that $M$ \emph{absorbs} $N$.

For a C$^*$-algebra $A$, let us denote by $\mathrm{W}(A)_+$ the classes of those elements in $M_{\infty}(A)_+$ which are not Cuntz equivalent to a projection. Note that, if $A$ has stable rank one, then $\mathrm{W}(A)_+$ absorbs $\mathrm{W}(A)$ (see, e.g. \cite{APT}). If now $X$ is a finite dimensional compact metric space, define
\[
\mathrm{Lsc}_{\mathrm{b}}(X,\mathrm{W}(A)_+)=\{f\in \mathrm{Lsc}_{\mathrm{b}}(X,\mathrm{W}(A))\mid f(X)\subseteq \mathrm{W}(A)_+\}\,.
\]
It is clear that $\mathrm{Lsc}_{\mathrm{b}}(X,\mathrm{W}(A)_+)$ absorbs $\mathrm{Lsc}_{\mathrm{b}}(X,\mathrm{W}(A))$.
\begin{theorem}
\label{cor:tiku}
Let $X$ be a finite dimensional, compact metric space, and let $A$ be a unital, simple, non-type I, ASH algebra with slow dimension growth. Then $\mathrm{K}_0^*(\mathrm{C}(X,A))$ is an interpolation group.
\end{theorem}
\begin{proof}
A description of $\mathrm{W}(\mathrm{C}(X,A))$ for the algebras in the hypothesis is given in \cite[Corollary 7.1]{Tiku} by means of pairs $(f,P)$ 
consisting of a lower semicontinuous function $f\in\mathrm{Lsc}_{\mathrm{b}}(X,\mathrm{W}(A))$, and a collection $P$, indexed over $[p]\in V(A)$,
of projection valued functions in $\mathrm{C}(f^{-1}([p]),A\otimes \mathcal K)$ modulo a certain equivalence relation.
If $f\in \mathrm{Lsc}_{\mathrm{b}}(X,\mathrm{W}(A)_+)$, then clearly $f^{-1}([p])=\emptyset$ for all $[p]\in V(A)$ thus notably simplifying the description of these elements. Namely,
there is only one pair of the form $(f,P_0)$, where $P_0$ does not depend on $f\in\mathrm{Lsc}_{\mathrm b}(X,\mathrm{W}(A)_+)$. In particular, the assignment $f\mapsto (f,P_0)$ defines an order-embedding $\mathrm{Lsc}_{\mathrm{b}}(X,\mathrm{W}(A)_+)\to \mathrm{W}(\mathrm{C}(X,A))$ whose image absorbs $\mathrm{W}(\mathrm{C}(X,A))$. As we also have that $\mathrm{Lsc}_{\mathrm{b}}(X,\mathrm{W}(A)_+)$ absorbs $\mathrm{Lsc}_{\mathrm{b}}(X,\mathrm{W}(A))$, we have by Proposition \ref{Grot} that 
\[
\mathrm{K}_0^*(\mathrm{C}(X,A))\cong \mathrm{G}(\mathrm{Lsc}_{\mathrm{b}}(X,\mathrm{W}(A)_+))\cong \mathrm{G}(\mathrm{Lsc}_{\mathrm{b}}(X,\mathrm{W}(A)))\,,
\]
as partially ordered abelian groups. Since $\mathrm{W}(A)$ is an interpolation semigroup (\cite[Proposition 5.4]{Tiku}) we conclude, using Corollary \ref{interC} and \cite[Lemma 4.2]{Per}, that $\mathrm{K}_0^*(\mathrm{C}(X,A))$  is an interpolation group.
\end{proof}

We close this section by analysing continuous fields over one-dimensional spaces. As in the previous results, we will need a representation of the Cuntz semigroup of such algebras, which in this case can be done in terms of continuous sections over a topological space. We recall the main definitions below (see \cite{ABP2} for a fuller account).

Let $X$ be a compact Hausdorff space, and denote by $\mathcal V_X$ the category of closed sets with non-empty interior, with morphisms given by inclusion. A \emph{$\Cu$-presheaf} over $X$ is a contravariant functor $\mathcal S\colon \mathcal V_X\to\Cu$. We denote by $S=\mathcal S(X)$ and by $\pi_V^W\colon \mathcal S(W)\to \mathcal S(V)$ the restriction maps (where $V\subseteq W$). A $\Cu$-presheaf is a \emph{$\Cu$-sheaf} if, whenever $V$ and $V'\in\mathcal V_X$ are such that $V\cap V'\in\mathcal{V}_X$, the map
\[\pi_{V}^{V\cup V'}\times\pi_{V'}^{V\cup V'}\colon 
\mathcal{S}(V\cup V')\to \{(f,g)\in \mathcal{S}(V)\times\mathcal{S}(V')\mid
\pi_{V\cap V'}^V(f)=\pi_{V\cap V}^{V'}(g)\},
\]
is bijective. We say that a $\Cu$-presheaf (respectively a $\Cu$-sheaf) is \emph{continuous} if for any decreasing sequence of
closed subsets $(V_i)_{i=1}^\infty$ with $\cap_{i=1}^\infty V_i=V\in\mathcal V_X$, the limit $\lim \mathcal{S}(V_i)$ is
isomorphic to $\mathcal{S}(V)$. We will assume from now on that all $\Cu$-presheaves and all $\Cu$-sheaves are continuous.

Given a $\Cu$-presheaf $\mathcal{S}$ over $X$, and $x\in X$, we define the \emph{fiber} of $\mathcal{S}$ at $x$ as  $S_x:=\lim_{x\in \mathring{V}}\mathcal{S}(V)$.  We will denote by $\pi_x\colon S\to S_x$ the natural maps, and also by $\pi_U$ the maps $\pi_U^X$. To ease the notation in the sequel, we shall refer to $\Cu$-presheaves or $\Cu$-sheaves as presheaves or sheaves, respectively. A (pre)sheaf is \emph{surjective} if all the restriction maps are surjective.

For a presheaf $\mathcal S$, we define $F_S=\sqcup_{x\in X} S_x$ and $\pi\colon F_S\to X$ by $\pi(s)=x$ if $s\in S_x$. A \emph{section} of $F_S$ is a map $f\colon X\to F_S$ such that $\pi f=\mathrm{id}_X$. Elements of $S$ induce sections as follows: given $s\in S$, we denote by $\hat s(x)=\pi_x(s)$, and it is clear that $\hat s$ defines a section of $F_S$. A section $f\colon X\to F_S$ is \emph{continuous} if the following holds: 
\begin{itemize}
\item[-] For all $x\in X$ and $a_x\in S_x$ such that $a_x\ll f(x)$,
there exist a closed set $V$ with $x\in\mathring{V}$
and $s\in S$ such that $\hat{s}(x)\gg a_{x}$ and
$\hat{s}(y)\ll f(y)$ for all $y\in V$.
\end{itemize}
We denote by $\Gamma(X,F_S)$ the set of all continuous sections, which becomes a partially ordered abelian semigroup when equipped with the pointwise order and addition. As shown in \cite[Theorem 3.10]{ABP2}, if $X$ is one dimensional and $\mathcal S\colon \mathcal V_X\to \Cu$ is a surjective sheaf over $X$ with $S$ countably based, then $\Gamma(X,F_S)$ belongs to the category $\Cu$. 

Our main interest is in the $\Cu$-sheaf determined by continuous fields. If $A$ is a continuous field over $X$ whose fibers have no $\mathrm{K}_1$-obstructions, and $\Cu_A$ denotes the sheaf given by $\Cu_A(V)=\Cu(A(V))$, then  the natural map $\Cu(A)\to\Gamma(X,F_{\Cu(A)})$ defined by $s\mapsto\hat s$ is an order-isomorphism in $\Cu$ (by \cite[Theorem 3.12]{ABP2}). 

\begin{proposition}
\label{prop:interpsheaf} Let $X$ be a one dimensional, compact metric space, and let $\mathcal S\colon\mathcal V_X\to \Cu$ be a surjective sheaf such that $S_x$ is an interpolation semigroup for each $x\in X$. Then $\Gamma(X,F_S)$ is also an interpolation semigroup.
\end{proposition}
\begin{proof}
We apply Lemma \ref{lemaprevi}, and so suppose that $f_i'\ll f_i\ll g_j$, for $i,j=1,2$. Given $x\in X$, there are elements $a_{i,x}\in S_x$ such that $f_i'(x)\ll a_{i,x}\ll f_i(x)$ for each $i$, and that satisfy condition (ii) in \cite[Proposition 3.2]{ABP2}. As $a_{i,x}\ll f_i(x)$, there are by continuity a closed neighborhood $V_x$ of $x$ with $x\in\mathring{V}_x$ and $s_i\ll s'_i\ll s''_i\in S$ (depending on $x$) such that $a_{i,x}\ll \hat{s}_i$ and $\hat{s}_i''(y)\ll f_i(y)$ for all $y\in V_x$. Now apply \cite[Proposition 3.2 (ii)]{ABP2} so there is a closed neighborhood $W_x\subseteq V_x$ (whose interior contains $x$) such that $f_i'(y)\leq \hat{s}_i(y)\ll \hat{s}_i'(y)\ll \hat{s}_i''(y)\ll f_i(y)$ for all $y\in W_x$.

At $x$, we have that $\hat{s}_i'(x)\ll g_j(x)$, so by the interpolation property assumed on $S_x$ and \cite[Lemma 3.3]{ABP2}, there are elements $c_x\ll c_x'$ in $S$ such that 
\[f_i'(x)\leq \hat{s}_i(x)\ll \hat{s}_i'(x)\ll \hat{c}_x(x)\ll \hat{c}_x'(x)\ll g_j(x)\,.
\]
 Since $c_x\ll c_x'$, we may apply \cite[Corollary 3.4]{ABP2} to find a closed subset $W'_x\subseteq\mathring{W}_x$ such that $\pi_{W'_x}(c_x)\ll {g_j}_{|W'_x}$ for each $j$. Since $s_i\ll s'_i$ for each $i$, another application of \cite[Corollary 3.4]{ABP2} yields a closed subset $W_x''\subseteq\mathring{W}_x$ such that $\pi_{W''_x}(s_i)\ll \pi_{|W''_x}(c_x)$ for each $i$.  We therefore conclude that $f_i'(y)\ll \hat{c}_x(y)\ll g_j(y)$ for all $y\in W_x'\cap W''_x$ and for all $i,j\in\{1,2\}$.
By compactness we obtain a finite cover $W_1,\ldots, W_n$ of $X$  and elements $c_1,\ldots,c_n\in S$ such that $f_i'(y)\ll\hat{c_i}(y)\ll g_j(y)$ for all $y\in W_i$. We now run the argument in \cite[Proposition 3.8]{ABP2} to patch the sections $\hat{c}_i$ into a continuous section $h\in\Gamma(X, F_S)$ such that $f_i\ll h\ll g_j$.
\end{proof}

\begin{theorem}\label{th:ctsfields}
Let $X$ be a one dimensional, compact metric space. Let $A$ be a continuous field over $X$ such that, for all $x\in X$, $A_x$ has stable rank one, trivial $K_1$, and is either of real rank zero, or simple and $\mathcal Z$-stable. Then $\mathrm{K}_0^*(A)$ is an interpolation group. 
\end{theorem}
\begin{proof}
As mentioned above, by \cite[Theorem 3.12]{ABP2}  we have an order-isomorphism between $\Cu(A)$ and $\Gamma(X,F_{\Cu(A)})$, and the latter is an interpolation semigroup by Proposition \ref{prop:interpsheaf}. Furthermore, $A$ has stable rank one by Theorem \ref{StableRankOfCtsFields}, and so $\mathrm{W}(A)$ is hereditary. Hence, $\mathrm{W}(A)$ will also be an interpolation semigroup (Lemma \ref{interpolation}) and $\mathrm{K}_0^*(A)$ is an interpolation group.
\end{proof}

\section{Structure of dimension functions}

In this section we apply the above results to confirm the conjectures of Blackadar and Handelman for certain continuous fields of C$^*$-algebras.

If $A$ is unital, the maps $d\colon \mathrm{W}(A)\to \mathbb{R}^+$ that respect addition, order, and satisfy $d(\la 1_A\ra)=1$ are called \emph{dimension functions}, and we denote the set of them by $\mathrm{DF}(A)$. In other words, $\mathrm{DF}(A)$ equals the set of states $\mathrm{St}(\mathrm{W}(A),\la 1_A\ra)$ on the semigroup $\mathrm{W}(A)$, which clearly agrees with 
$\mathrm{St}(\mathrm{K}_0^*(A),\mathrm{K}_0^*(A)^+,[1_A])$.

\begin{theorem}\label{cuntzse}
Let $X$ be a finite dimensional, compact metric space, and let  $A$ be a separable, unital C$^*$-algebra. Then $\mathrm{DF}(A)$ is a Choquet simplex in the following cases:
\begin{enumerate}[{\rm (i)}]
\item $\dim X\leq 1$ and $A$ is a continuous field such that, for all $x\in X$, $A_x$ has stable rank one, trivial $\mathrm{K}_1$ and is either of real rank zero or else simple and $\mathcal Z$-stable.
\item $X$ is an arc-like space and $A=\mathrm{C}(X,B)$ where $B$ is simple, with real rank zero, and has finite radius of comparison, or else $B$ is simple and $\mathcal Z$-stable.
\item $\dim X\leq 2$, $\check{\mathrm{H}}^2(X,\mathbb Z)=0$, and $A=\mathrm{C}(X,B)$ with $B$ an AF-algebra.
\item $A=\mathrm{C}(X,B)$, where $B$ is a non-type I, simple, ASH algebra with slow dimension growth.
\end{enumerate}
\end{theorem}
\begin{proof}
By the results of Section \ref{sec:lsc}, $\mathrm{K}_0^*(A)$ is an interpolation group in all  the cases. Then, by \cite[Theorem 10.17]{Goo}, $\mathrm{DF}(A)$ is a Choquet simplex.
\end{proof}

\begin{proposition}\label{lsc}
Let $X$ and $Y$ be compact Hausdorff spaces. Put
\[
\mathrm{G}_{\mathrm{b}}(X,Y)=\{f\colon X\times Y\to \mathbb R\mid f=g-h\text{ with }g,h\in
\mathrm{Lsc}_{\mathrm{b}}(X\times Y)^{++}\}\,. 
\]
Then
\begin{enumerate}[{\rm (i)}]
\item $\mathrm{G}_{\mathrm{b}}(X,Y)$, equipped with the pointwise order, is a partially ordered abelian group.
 \item For any $f\in
\mathrm{\mathrm{Lsc}}_{\mathrm{b}}(X,\mathrm{Lsc}_{\mathrm{b}}(Y)^{++})$, the map $\tilde{f}\colon X\times Y\to \mathbb{R}^+$, defined by
$\tilde{f}(x,y)=f(x)(y)$, is lower semicontinuous.
\item The map $\beta\colon \mathrm{G}(\mathrm{Lsc}_{\mathrm{b}}(X,\mathrm{Lsc}_{\mathrm{b}}(Y)^{++}))\to
\mathrm{G}_{\mathrm{b}}(X,Y)$ defined by 
\[
\beta([f]-[g])=\tilde{f}-\tilde{g}
\]
is an order-embedding.
\end{enumerate}
\end{proposition}
\begin{proof}
(i): This is trivial.

(ii): We have to show that the set $U_{\alpha}=\{(x,y)\mid
f(x)(y)>\alpha\}$ is open for all $\alpha>0$.

Fix $(x_0,y_0)\in U_{\alpha}$. Since $f(x_0)(y_0)>\alpha$, we may consider
$f(x_0)(y_0)>\alpha+\epsilon'>\alpha+\epsilon>\alpha$ for some
$\epsilon,\epsilon'> 0$. Since $f(x_0)$ is lower semicontinuous, there exists
an open set $V'_{y_0} \subseteq Y$ containing $y_0$ such that $f(x_0)(y)>
\alpha+\epsilon$ for all $y\in V'_{y_0}$. Now, as $Y$ is compact, $f(x_0)$ is
bounded away from zero and we find $0<\epsilon_0< \alpha$ such that
$f(x_0)(y)>\epsilon_0$ for all $y\in Y$.

Let $V_{y_0}$ be an open neighborhood of $y_0$ such that $V_{y_0}\subseteq \overline{V}_{y_0}\subseteq V'_{y_0}$. Define $g\in \mathrm{Lsc}_{\mathrm{b}}(Y)^{++}$ by
$g(y)=\alpha+\epsilon\text{ when } y\in V_{y_0}$ and
$g(y)=\epsilon_0<\alpha$ otherwise. Observe that, by the way we have chosen $V_{y_0}$ and the construction of $g$, for
every $y\in Y$, there exists $U_y$ containing $y$ and
$\lambda_y\in\mathbb{R}^+$ such that
$g(y')\leq \lambda_y <f(x_0)(y')$ whenever $y'\in U_y$. This implies that
$g\ll f(x_0)$ in $\mathrm{Lsc}_{\mathrm{b}}(Y)$.

Since $f$ is lower semicontinuous, $\{x\in X\mid f(x)\gg g\}$ is an open set containing $x_0$. Thus,
we may find an open set $U_{x_0}$ such
that $x_0\in U_{x_0}$ and $f(x)\gg g$ for all $x\in U_{x_0}$. Now, for
$(x,y)$ in the open set $U_{x_0}\times V_{y_0}\subseteq X\times Y$ we have
$\tilde{f}(x,y)=f(x)(y)>g(y)=\alpha+\epsilon>\alpha$.

(iii): We first need to check that $\beta$ is well-defined. Suppose that  $[f]-[g]=
[f']-[g']$ in $\mathrm{G}(\mathrm{Lsc}_{\mathrm{b}}(X,\mathrm{Lsc}_{\mathrm{b}}(Y)^{++}))$. Then there is $h$
such that $f+g'+h=f'+g+h$. Since $h(x)$ is bounded for every $x$, we obtain
$f(x)(y)+g'(x)(y)=f'(x)(y)+g(x)(y)$ for all $x$ and $y$, and so
$f(x)(y)-g(x)(y)=f'(x)(y)-g'(x)(y)$. By (ii), it is clear that
$\beta([f]-[g])\in G_\mathrm{b}(X,Y)$, and that it is a group homomorphism. If $[f]-[g]\in \mathrm{G}(\mathrm{Lsc}_{\mathrm{b}}(X,\mathrm{Lsc}_{\mathrm{b}}(Y)^{++}))$, then $g\leq f$, if and only if $g(x)(y)\leq f(x)(y)$ for each $x$ and $y$, proving that $\beta$ is an order-embedding. 
\end{proof}

 As customary, for a unital C$^*$-algebra $A$ we we denote the set of normalized quasi traces on $A$ by $\mathrm{QT}(A)$, and the set of normalized tracial states by $\mathrm{T}(A)$. Recall that for an exact C$^*$-algebra $\mathrm{QT}(A)=\mathrm{T}(A)$ (\cite{haag}). We also denote the set of extreme points of a convex set $K$ by $\partial_eK$.  

Although the result below might be well-known to experts, we provide a proof for completeness (with thanks to Nate Brown).
\begin{lemma}
\label{nate}
 Let $X$ be a compact Hausdorff space and let $A$ be a unital C$^*$-algebra. Then there exists a homeomorphism between $\partial_e\mathrm{T}(\mathrm{C}(X,A))$ and $X\times \partial_e\mathrm{T}(A)$. Moreover, if $\tau\in\partial_e\mathrm{T}(\mathrm{C}(X,A))$ corresponds to $(x,\tau_A)$, then $d_{\tau}(b)=d_{\tau_A}(b(x))$ for any $b\in M_{\infty}(\mathrm{C}(X,A))_+$.
\end{lemma}

\begin{proof}
Recall that a normalized trace on a unital C$^*$-algebra is extremal if, and only if, the weak closure of its corresponding GNS-representation is a factor (i.e. has trivial center), see, e.g. \cite[Theorem 6.7.3]{Dixmier}. Now identify $\mathrm{C}(X,A)$ with $B:=\mathrm{C}(X)\otimes A$. Let $\tau\in\partial_e\mathrm{T}(B)$ and let $(\pi_\tau,\mathcal H_\tau,v)$ be the GNS-triple associated to $\tau$, and we know that $\pi_\tau(B)''$ is a factor.

Since $\mathrm{C}(X)\otimes 1_A$ is in the center of $B$, we have that $\pi_\tau(\mathrm{C}(X)\otimes 1_A)$ is in the center of $\pi_\tau(B)''$, whence $\pi_\tau(\mathrm{C}(X)\otimes 1_A)=\mathbb{C}$. Thus, the restriction of $\pi_\tau$ to $\mathrm{C}(X)\otimes 1_A$ corresponds to a point evaluation $\mathrm{ev}_{x_0}$ for some $x_0\in X$.

Next,
\[
\tau(f\otimes a)=\langle \pi_\tau(f\otimes a)v,v\rangle=\langle \mathrm{ev}_{x_0}(f)\pi_\tau(1\otimes a)v,v\rangle=f(x_0)\langle \pi_\tau(1\otimes a)v,v\rangle=f(x_0)\tau(1\otimes a)\,,
\] 
for all $f\in \mathrm{C}(X)$ and $a\in A$. Therefore $\tau=ev_{x_0}\otimes \tau_A$ where $\tau_A$ is the restriction of $\tau$ to $1\otimes A$. Note that $\tau_A$ is extremal as $\tau$ is.
%Note that as $\pi_\tau(B)''$ is a factor it implies that $\pi_{\tau_A}(A)''$ is also a factor, so $\tau_A$ is extremal.

We thus have a map $ \psi\colon \partial_e\mathrm{T}(B)\to \partial_e\mathrm{T}(\mathrm{C}(X))\times \partial_e\mathrm{T}(A)$ defined by $\psi(\tau)=(ev_{x_0}, \tau_A)$, which is easily seen to be a homeomorphism.

Now identify $M_n(\mathrm{C}(X,A))$ with $\mathrm{C}(X,M_n(A))$ and let $b\in \mathrm{C}(X,M_n(A))_+$. Let $\tau\in\partial_e\mathrm{T}(B)$ and $\psi(\tau)=(x,\tau_A)$. Then
\[
d_{\tau}(b)=\lim_{k\to\infty}\tau(b^{1/k})=\lim_{k\to\infty}\tau_A(b^{1/k}(x))=d_{\tau_A}(b(x))\,.
\]
\end{proof}

Given $\tau\in \mathrm{QT}(A)$ and $a\in M_\infty(A)_+$, we may construct
 $$d_\tau(a)=\lim_{n\to\infty}\tau(a^{1/n})\,.$$
It turns out that the above map only depends on the Cuntz equivalence class of $a$, and that it defines a lower semicontinuous state on $\mathrm{W}(A)$ (see \cite{BH,Cu}). These states are called \emph{lower
semiconti\-nuous dimension functions}, and we denote them by $\textrm{LDF}(A)$.

If $K$ is a compact convex set, we shall denote by $\mathrm{LAff}_{\mathrm{b}}(K)^{++}$ the
semigroup of (real-valued) bounded, strictly positive, lower semicontinuous and
affine functions on $K$. This is a subsemigroup of the group $\mathrm{Aff}_{\mathrm{b}}(K)$ of all real-valued, bounded affine functions defined on $K$. Now, given a C$^*$-algebra $A$, we may define a semigroup homomorphism
\[
\varphi\colon \mathrm{W}(A)_+\to \mathrm{LAff}_{\mathrm{b}}(\mathrm{QT}(A))\,,
\]
by $\varphi(\langle a\rangle)(\tau)=d_{\tau}(a)$ (see, e.g. \cite{APT}, \cite{pertoms}). For ease of notation, we shall denote $\varphi(\langle a\rangle)=\hat{a}$. Notice that, if $A$ is simple, then $\hat{a}\in\mathrm{LAff}_{\mathrm{b}}(\mathrm{QT}(A))^{++}$ if $a$ is non-zero.

Observe also that there is an ordered morphism $\alpha\colon \mathrm{W}(\mathrm{C}(X,A))\to \mathrm{Lsc}_{\mathrm b}(X,\mathrm{W}(A))$, given by $\alpha(\langle b\rangle)(x)=\langle b(x)\rangle$.

\begin{proposition}
\label{prop:embedding} Let $X$ be a compact Hausdorff space, and let $A$ be a separable, exact, infinite dimensional, simple, unital, C$^*$-algebra with strict comparison and such that $\mathrm{T}(A)$ is a Bauer simplex. Then there is an order-embedding
\[
\mathrm{G}(\mathrm{Lsc}_{\mathrm{b}}(X, \mathrm{W}(A)))\to\mathrm{Aff}_{\mathrm{b}}(\mathrm{T}(\mathrm{C}(X,A)))\,.
\]
Moreover, given $b\in \mathrm{C}(X,M_n(A))_+$, this map sends the class of the function $\alpha(\langle b\rangle)$ to $\hat b$. 
\end{proposition}
\begin{proof}
Since $\mathrm{Lsc}_{\mathrm{b}}(X,\mathrm{W}(A)_+)$ absorbs $\mathrm{Lsc}_{\mathrm{b}}(X,\mathrm{W}(A))$, there is by Lemma \ref{Grot} an order-iso\-mor\-phism between $\mathrm{G}(\mathrm{Lsc}_{\mathrm{b}}(X,\mathrm{W}(A)))$ and $\mathrm{G}(\mathrm{Lsc}_{\mathrm{b}}(X,\mathrm{W}(A)_+))$. In fact, if we take $\langle a\rangle\in \mathrm{W}(A)_+$, and let $v\colon X\to \mathrm{W}(A)_+$ be the function defined as $v(x)=\langle a\rangle$, the previous isomorphism takes $[\alpha(\langle b\rangle)]$ to $[\alpha(\langle b\rangle)+v]-[v]$. Next, as $A$ has strict comparison, the semigroup homomorphism $\varphi$ defined previous to this proposition is an order-embedding (see \cite[Theorem 4.4]{pertoms}) and thus induces 
\[
\mathrm{G}(\mathrm{Lsc}_{\mathrm{b}}(X,\mathrm{W}(A)_+))\to \mathrm{G}(\mathrm{Lsc}_{\mathrm{b}}(X,\mathrm{LAff}_{\mathrm{b}}(\mathrm{T}(A))^{++}))\,,
\]
which is also an order-embedding, and takes $[\alpha(\langle b\rangle)+v]-[v]$ to $[\tilde{\varphi}(\alpha(\langle b\rangle))+\hat{a}]-[\hat{a}]$, where we identify $\hat{a}$ with a constant function and $\tilde{\varphi}(\alpha(\langle b\rangle))(x)=\widehat{b(x)}$.
Now, since $\mathrm{T}(A)$ is a Bauer simplex, the restriction to the extreme boundary yields a semigroup isomorphism $r\colon \mathrm{LAff}_{\mathrm{b}}(\mathrm{T}(A))^{++}\cong \mathrm{Lsc}_{\mathrm{b}}(\partial_e \mathrm{T}(A))^{++}$ (see, e.g. \cite[Lemma 7.2]{Goo2}). Combining these observations with condition (iii) in Proposition \ref{lsc}, we obtain an order-embedding
\[
\mathrm{G}(\mathrm{Lsc}_{\mathrm{b}}(X,\mathrm{Lsc}_{\mathrm{b}}(\partial_e\mathrm{T}(A))^{++}))\to \mathrm{G}_{\mathrm{b}}(X,\partial_e\mathrm{T}(A))\,,
\]
that sends $[r(\tilde{\varphi}(\alpha(\langle b\rangle))+\hat{a})]-[r(\hat{a})]$ to $r(\tilde{\varphi}(\alpha(\langle b\rangle))+\hat{a})^{\sim}-r(\hat{a})^{\sim}$, which equals the function $(x,\tau_A)\mapsto d_{\tau_A}(b(x))$. Finally, upon identifying the compact space $X\times\partial_e\mathrm{T}(A)$ with $\partial_e\mathrm{T}(\mathrm{C}(X,A))$ (by Lemma \ref{nate}), a second usage of \cite[Lemma 7.2]{Goo2} allows us to order-embed $\mathrm{G}_{\mathrm{b}}(X,\partial_e\mathrm{T}(A))$ into $\mathrm{Aff}_{\mathrm{b}}(\mathrm{T}(\mathrm{C}(X,A)))$, and the  map $(x,\tau_A)\mapsto d_{\tau_A}(b(x))$ is sent to $\hat{b}$, as desired.
\end{proof}

\begin{theorem}\label{A}
Let $X$ be a finite dimensional, compact metric space, and let  $A$ be a unital, separable, infinite dimensional and exact C$^*$-algebra of stable rank one such that  $\mathrm{T}(A)$ is a Bauer simplex. Then $\mathrm{LDF}(\mathrm{C}(X,A))$ is dense in $\mathrm{DF}(\mathrm{C}(X,A))$ in the following cases:
\begin{enumerate}[{\rm (i)}]
\item $\dim X\leq 1$ and $A$ is simple, $\mathrm{K}_1(A)=0$ and $A$ has strict comparison. 
\item $X$ is arc-like, $A$ is simple, has real rank zero,  and strict comparison.
\item $\dim X\leq 2$ and $\check{\mathrm{H}}^2(X,\mathbb Z)=0$, with $A$ an AF-algebra.
\item $A$ is a non-type I, simple, unital ASH algebra with slow dimension growth.
\end{enumerate}

\end{theorem}
\begin{proof}
(i): By \cite[Theorem 3.4]{APS}, we have that $\Cu(\mathrm{C}(X,A))$ and $\mathrm{Lsc}(X,\Cu(A))$ are order-isomorphic, and $\mathrm{W}(\mathrm{C}(X,A))$ is hereditary as it has stable rank one by Theorem \ref{StableRankOfDimOne}. It follows easily from this that $\mathrm{W}(\mathrm{C}(X,A))$ is order-isomorphic to $\mathrm{Lsc}_{\mathrm{b}}(X,\mathrm{W}(A))$. We may apply Proposition \ref{prop:embedding} so that $\mathrm{K}_0^*(\mathrm{C}(X,A))$ is order-isomorphic to a (pointwise ordered) subgroup $G$ of $\mathrm{Aff}_{\mathrm b}(\mathrm{T}(\mathrm{C}(X,A)))$ in such a way that $[b]$ is mapped to $\hat{b}$, and in particular $[1]$ is sent to the constant function $1$.

We apply now the same argument as in \cite[Theorem 6.4]{BPT}, which we sketch for convenience. If $d\in \mathrm{DF}(\mathrm{C}(X,A))$, then it can be identified with  a normalized state (at $1$) on $G$. By \cite[Lemma 6.1]{BPT}, there is a net of traces $(\tau_i)$ in $\mathrm{T}(\mathrm{C}(X,A))$ such that $d(s)=\lim_i s(\tau_i)$ for any $s\in G$. In particular, $d([b])=\lim_i \hat{b}(\tau_i)=d_{\tau_i}(b)$ for $b\in M_{\infty}(\mathrm{C}(X,A))_+$.

(ii): This case uses the same arguments as (i), replacing \cite[Theorem 3.4]{APS} by Proposition \ref{lem:cuntzse} and its proof.

(iii): Proceed as in case (i),  using \cite[Corollary 3.6]{APS} instead of \cite[Theorem 3.4]{APS} and Remark \ref{rem:hereditary}.

(iv): As in the proof of Theorem \ref{cor:tiku}, we see that $\mathrm{K}_0^*(\mathrm{C}(X,A))\cong \mathrm{G}(\mathrm{Lsc}_{\mathrm b}(X,\mathrm{W}(A)))$ as ordered groups, and then we may use the same argument as in case (i).
\end{proof}

\section*{Acknowledgements}

This work has been partially supported by a MICIIN grant (Spain) through Project MTM2011-28992-C02-01, and by the Comissionat per Universitats i Recerca de la Generalitat de Catalunya. The fourth named author has received support from the DFG (Germany), PE 2139/1-1.

\end{document}